\documentclass[12 pt]{article}
\usepackage[centertags]{amsmath}
\usepackage{amssymb}
\usepackage{amsthm}
\usepackage{newlfont}
\usepackage{amsfonts}

\usepackage{graphicx}
\usepackage{color}
\usepackage[colorlinks]{hyperref}
\usepackage{mathptmx}      % use Times fonts if available on your TeX system
\newtheorem{thm}{Theorem}
\newtheorem{cor}{Corollary}
\newtheorem{lem}{Lemma}
\newtheorem{prop}{Proposition}
\newtheorem{conjec}{Conjecture}
\newtheorem{defn}{Definition}

\newcommand{\w}{\omega}
\newcommand{\ra}{\rightarrow}

\newcommand{\sub}{\subseteq}
\newcommand{\mo}{$L=x_1x_2...x_n$}
\newcommand{\bks}{\backslash}

\newcommand{\degr}{$d^+_D(v)\leq d^{++}_D(v)$}

\begin{document}

\begin{center}\Large \textbf{A Remark on the Second Neighborhood Problem}
\end{center}
\begin{center}\Large Salman Ghazal
\end{center}

\begin{abstract}
Seymour's second neighborhood conjecture states that every simple digraph (without digons) has a vertex whose first out-neighborhood is at most as large as its second out-neighborhood. Such a vertex is said to have the second neighborhood property (SNP). We define "good" digraphs and prove a statement that implies that every feed vertex of a tournament has the SNP. In the case of digraphs missing a matching, we exhibit a feed vertex with the SNP by refining a proof due to Fidler and Yuster and using good digraphs. Moreover, in some cases we exhibit two vertices with SNP.
%\keywords{Oriented graph \and Out-neighborhood \and Second out-neighborhood \and out-degree \and Weighted digraph \and Matching}
\end{abstract}

\section{Introduction}
\label{intro}
In this paper, a digraph $D$ is a couple of two sets $(V, D)$, where $E\sub V\times V$. $V$ and $E$ are the vertex set and edge set of $D$ and denoted by $V(D)$ and $E(D)$ respectively. An oriented graph is a digraph that contains neither loops nor digons.
If $K\subseteq V(D)$ then the induced restriction of $D$ to $K$ is denoted by $D[K]$. As usual, $N^{+}_D(v)$ (resp. $N^{-}_D(v)$)
denotes the (first) out-neighborhood (resp. in-neighborhood) of a vertex $v\in V$. $N^{++}_D(v)$ (resp. $N^{--}_D(v)$) denotes the second
\textit{out-neighborhood} (\textit{in-neighborhood}) of $v$, which is the set of vertices that are at distance 2 from $v$ (resp. to $v$). We also denote
$d^{+}_D(v)=|N^{+}_D(v)|$, $d^{++}_D(v)=|N^{++}_D(v)|$, $d^{-}_D(v)=|N^{-}_D(v)|$ and $d^{--}_D(v)=|N^{--}_D(v)|$.
We omit the subscript if the digraph is clear from the context.
For short, we write $x\rightarrow y$ if the arc $(x,y)\in E$.
A vertex $v\in V(D)$ is called \textit{whole} if $d(v):= d^{+}(v) + d^{-}(v)=|V(D)|-1$, otherwise $v$ is non whole. A \textit{sink} $v$ is a vertex with $d^{+}(v)=0$.
For $x,y\in V(D)$, we say $xy$ is a \textit{missing edge} of $D$ if neither $(x,y)$ nor $(y,x)$ are in $E(D)$.
The\textit{ missing graph} $G$ of $D$ is
the graph whose edges are the missing edges of $D$ and whose vertices are the non whole vertices of $D$.
In this case, we say that $D$ is \textit{missing} $G$. So, a tournament does not have missing edges.\\

\par A vertex $v$ of $D$ is said to have the \textit{second neighborhood property} (SNP) if \degr %$d^+(v)\leq d^{++}(v)$
 . In 1990, Seymour conjectured the following:\\

\begin{conjec}
\textbf{(Seymour's Second Neighborhood Conjecture (SNC))}\cite{dean}
\hspace{5pt}Every oriented graph has a vertex with the SNP.
\end{conjec}

In 1996, Fisher \cite{fisher} solved the SNC for tournaments by using a certain probability distribution
on the vertices. Another proof of Dean's conjecture was established in 2000 by Havet and Thomass\'{e} \cite{m.o.}. Their short proof uses a tool called median orders. Furthermore, they have proved that if a tournament has no sink vertex then there are at least two vertices
with the SNP.\\

\par Let $D=(V,E)$ be a digraph (vertex) weighted by a positive real valued function \hskip.1cm$\omega :V\rightarrow \mathcal{R_+}$. The couple $(D, \w)$ (or simply $D$) is called a \textit{weighted digraph}. The weight
of an arc $e=(x,y)$ is $\omega(e):=\omega(x).\omega(y)$ . The weight of a set of vertices (resp. edges) is the sum of the weights of its members.
We say that a vertex $v$ has the \textit{weighted SNP }if $\omega(N^+(v))\leq \omega(N^{++}(v))$. It is known that the SNC is equivalent to its weighted
version:\emph{ Every weighted oriented graph has a vertex with the weighted SNP}.\\

\par
Let $L=v_1v_2...v_n$  be an ordering of the vertices of a weighted digraph $(D,\omega)$.
An arc $e=(v_i,v_j)$ is \textit{forward} with respect to $L$ if $i<j$. Otherwise $e$ is a \textit{backward} arc. A \textit{weighted median order} $L=v_1v_2...v_n$ of $D$ is an order of the vertices of $D$ that maximizes the weight of the set of forward arcs of $D$, i.e., the set $\{(v_i,v_j)\in E(D); i<j\}$. In other words, $L=v_1v_2...v_n$ is a weighted median order of $D$ if $\omega(L)=max\{\omega(L'); L'$ is an ordering of ther vertices of $D\}$.
In fact, the weighted median order $L$ satisfies the \textit{feedback property}: For all $1\leq i\leq j\leq n:$
$$ \omega ( N^{+}_{D[i,j]}(v_i) ) \geq  \omega ( N^{-}_{D[i,j]}(v_i) ) $$
and
$$ \omega ( N^{-}_{D[i,j]}(v_j) ) \geq  \omega ( N^{+}_{D[i,j]}(v_j) ) $$
where $[i,j]:=\{v_i,v_{i+1}, ...,v_j\}$.\\

 Indeed, suppose to the contrary that $ \omega ( N^{+}_{D[i,j]}(v_i) ) <  \omega ( N^{-}_{D[i,j]}(v_i) ) $. Consider the order $L'=v_1...v_{i-1}v_{i+1}...v_jv_iv_{j+1}...v_n$ obtained from $L$ by inserting $v_i$ just after $v_j$. Then we have:
 $$\omega(L')=\omega(L)+\omega(\{(v_k,v_i)\in E(D); i\leq k\leq j\})-\omega(\{(v_i,v_k)\in E(D); i\leq k\leq j\})$$ $$=\omega(L)+\omega(v_i).\omega ( N^{-}_{D[i,j]}(v_i) )-\omega(v_i).\omega ( N^{+}_{D[i,j]}(v_j) )$$ $$=\omega(L)+\omega(v_i).(\omega ( N^{-}_{D[i,j]}(v_i) )-\omega ( N^{+}_{D[i,j]}(v_j) ))>\omega(L),$$ which contradicts the maximality of $\omega(L)$.\\

It is also known that if we reverse the orientation of a backward arc $e=(v_i,v_j)$ of $D$ with respect to $L$, then $L$ is again a weighted median order of the new weighted digraph $D'=D-(v_i,v_j)+(v_j,v_i)$.

\par When $\omega = 1$, we obtain
     the definition of median orders of a digraph (\cite{m.o.,fidler}).\\

Let $L=v_1v_2...v_n$ be a weighted median order. Among the vertices not in $N^+(v_n)$ two types are distinguished: A vertex $v_j$ is \textit{good} if there is $i\leq j$ such that
$v_n\ra v_i\ra v_j$, otherwise $v_j$ is a \textit{bad vertex}. The set of good vertices of $L$ is denoted by $G_L^D$ \cite{m.o.} ( or $G_L$ if there is no confusion ).
Clearly, $G_L \sub N^{++}(v_n)$. The last vertex $v_n$ is called a feed vertex of $(D,\omega)$.\\

In 2007, Fidler and Yuster \cite{fidler} proved that SNC holds for oriented graphs missing a matching.
They have used median orders and another tool called the dependency digraph. However, there proof does not
guarantee that the vertex found to have the SNP is a feed vertex.\\

\par In 2012, Ghazal also used the notion of weighted median order to prove the weighted SNC for digraphs missing a generalized star. As a corollary, the weighted version holds for digraphs missing a star, complete graph or a sun \cite{gs}. He also used the dependency digraph to prove SNC for other classes of oriented graphs \cite{contrib}.\\

\par We say that a missing edge $x_1y_1$ \textit{loses to} a missing edge $x_2y_2$ if:
$x_1\rightarrow x_2$, $y_2\notin N^{+}(x_1)\cup N^{++}(x_1)$, $y_1\rightarrow y_2$ and $x_2\notin N^{+}(y_1)\cup N^{++}(y_1)$.
The \textit{dependency digraph} $\Delta$ of $D$ is defined as follows: Its vertex set consists of all the missing
edges and $(ab,cd)\in E(\Delta)$ if $ab$ loses to $cd$ \cite{fidler,contrib}. Note that $\Delta$ may contain digons.\\

\begin{defn}\cite{gs}
In a digraph D, a missing edge $ab$ is called a \emph{good missing edge} if:\\
$(i)$   $(\forall v \in V\backslash\{a,b\})[(v\rightarrow a)\Rightarrow(b\in N^{+}(v)\cup N^{++}(v))]$ or\\
$(ii)$ $(\forall v \in V\backslash\{a,b\})[(v\rightarrow b)\Rightarrow(a\in N^{+}(v)\cup N^{++}(v))]$.\\
If $ab$ satisfies $(i)$ we say that $(a,b)$ is a convenient orientation of $ab$.\\
If $ab$ satisfies $(ii)$ we say that $(b,a)$ is a convenient orientation of $ab$.
\end{defn}

We will need the following observation:
\begin{lem} \cite{fidler}
 Let $D$ be an oriented graph and let $\Delta$ denote its dependency digraph. A missing edge $ab$ is good
if and only if its in-degree in $\Delta$ is zero.
\end{lem}

In the next section, we will define good median orders and good digraphs and prove a statement which implies
that every feed vertex of a weighted tournament has the weighted SNP. In the last section, we refine the proof of Fidler and Yuster and use good median orders to exhibit a feed vertex with the SNP in the case of oriented graphs missing a matching.

\section{Good median orders}
\label{sec:1}
Let $D$ be a (weighted) digraph and let $\Delta$ denote its dependency digraph. Let $C$ be a connected component of $\Delta$.
Set $K(C)=$ $\{u\in V(D);$ there is a vertex $v$ of $D$ such that $uv$ is a missing edge and belongs to $C$ $\}$. The \emph{interval
graph of $D$}, denoted by $\mathcal{I}_D$ is defined as follows. Its vertex set consists of the connected components
of $\Delta$ and two vertices $C_1$ and $C_2$ are adjacent if $K(C_1)\cap K(C_2)\neq \phi$. So  $\mathcal{I}_D$ is the
intersection graph of the family $\{K(C);C$ is a connected component of $\Delta$ $\}$. Let $\xi$ be a connected component of $\mathcal{I}_D$.
We set $K(\xi)=\displaystyle\cup _{C\in \xi}K(C)$. Clearly, if $uv$ is a missing edge in $D$ then there is a unique connected component $\xi$ of
$\mathcal{I}_D$ such that $u$ and $v$ belong to $K(\xi)$. For $f\in V(D)$, we set $J(f)=\{f\}$ if $f$ is a whole vertex, otherwise $J(f)=K(\xi)$,
where $\xi$ is the unique connected component of $\mathcal{I}_D$ such that $f\in K(\xi)$. Clearly, if $x\in J(f)$ then $J(f)=J(x)$ and
if $x\notin J(f)$ then $x$ is adjacent to every vertex in $J(f)$.\\

Let $L=x_{1}\cdots x_{n}$ be a  (weighted) median order of a digraph $D$. For $i<j$, the sets
$[i,j]:=[x_i,x_j]:=\{x_i,x_{i+1},...,x_j\}$ and  $]i,j[=[i,j]\bks \{x_i,x_j\}$
are called \textit{intervals} of $L$. We recall that $K\subseteq V(D)$ is an \textit{interval of $D$} if for every $u,v\in K$ we have:
$N^+(u)\backslash K =N^+(v)\backslash K$ and $N^-(u)\backslash K =N^-(v)\backslash K$. The following shows a relation
between the intervals of $D$ and the intervals of $L$.

\begin{prop}
Let $\mathcal{I}=\{I_1,...,I_r\}$ be a set of pairwise disjoint intervals of $D$. Then for every weighted median order $L$ of $D$,
there is a weighted median order $L'$ of $D$ such that: $L$ and $L'$ have the same feed vertex and every interval in $\mathcal{I}$
is an interval of $L'$.
\end{prop}

\begin{proof}
Let \mo be a weighted median order of a weighted digraph $(D,\w)$ and let $\mathcal{I}=\{I_1,...,I_r\}$ be a set of pairwise disjoint intervals of $D$.
We will use the feedback property to prove it. Suppose $a,b\in I_1$ with
$a=x_i$, $b=x_j$, $i<j$ and $[x_i,x_j]\cap I_1=\{x_i,x_j\}$. Since $I_1$ is an interval of $D$, we have $N^+_{]i,j[}(x_i)=N^+_{]i,j[}(x_j)$
and $N^-_{]i,j[}(x_i)=N^-_{]i,j[}(x_j)$. So,
$\w(N^-_{]i,j[}(x_i))$ $\leq\w(N^+_{]i,j[}(x_i))=N^+_{]i,j[}(x_j)\leq \w(N^-_{]i,j[}(x_j))=\w(N^-_{]i,j[}(x_i))$, where
the two inequalities are by the feedback property. Whence, all the quantities in the previous statement are equal. In particular,
$\w(N^+_{]i,j[}(x_i))= \w(N^-_{]i,j[}(x_i))$. Let $L_1$ be the enumeration $x_1...x_{i-1}x_{i+1}...x_{j-1}x_{x_i}x_{j}x_{j+1}...x_{n}$.
Then $\w(L_1)=\w(L) + \w(N^-_{]i,j[}(x_i)) - \w(N^+_{]i,j[}(x_i))=\w(L)$. Thus, $L_1$ is a weighted median order of $D$.
By successively repeating this argument, we obtain a weighted median order in which $I_1$ is an interval of $L$. Again,
by successively repeating the argument for each $I\in \mathcal{I}$, we obtain the desired order.
\end{proof}

We say that $D$ is \textit{good digraph} if the sets $K(\xi)$'s are intervals of $D$. By the previous proposition, every good digraph has a (weighted) median order $L$
such that the $K(\xi)$'s form intervals of $L$. Such an enumeration is called a \textit{good (weighted) median order} of the good digraph $D$.\\

\begin{thm}\label{prince}
Let $(D,\omega)$ be a good weighted oriented graph and let L be a good weighted median order of $(D,\omega)$, with feed vertex say f.
Then for every $x\in J(f)$, we have $\omega(N^+(x)\backslash J(f))\leq\omega(G_L\backslash J(f))$.
So if $x$ has the weighted SNP in $(D[J(f)], \w)$, then it has the weighted SNP in $D$.
\end{thm}

\begin{proof}
The proof is by induction $n$, the number of vertices of $D$. It is trivial for $n=1$. Let $L=x_1...x_n$ be a good weighted median order of ($D,\omega$). Set $f=x_n$, $J(f)=[x_t,x_n]$, $L_1=x_1...x_t$ and $D_1=D[x_1,x_t]$. Then $(D_1, \w)$ is a good weighted oriented graph and $L_1$ is a good weighted median order of $(D_1, \omega)$ in which $J(x_t)=\{x_t\}$. Suppose that $t<n$. Then by the induction hypothesis, $\w (N^+_{D_1}(x_t))\leq \w(G_{L_1})$. However, $J(f)$ is an interval of $D$, then for every $x\in J(f)$, we have $\omega(N^+(x)\backslash J(f))=\w (N^+(x_t)\backslash J(f))=\w (N^+_{D_1}(x_t))\leq \w(G_{L_1})=\omega(G_L\backslash J(f))$.
Now suppose that $t=n$. If $L$ does not have any bad vertex
then $N^-(x_n)=G_L$. Whence, $\omega(N^+(x_n))\leq\omega(N^-(x_n))=\omega(G_L)$ where the inequality is by the feedback property.
Now suppose that $L$ has a bad vertex and let $i$ be the smallest such that $x_i$ is bad. Since $J(x_i)$ is an interval of $D$ and $L$,
then every vertex in $J(x_i)$ is bad and thus $J(x_i)=[x_i,x_p]$ for some $p<n$.
For $j<i$, $x_j$ is either an out-neighbor of $x_n$ or a good vertex, by definition of $i$.
Moreover, if $x_j\in N^+(x_n)$ then $x_j\in N^+(x_i)$. So $N^+(x_n)\cap [1,i]\subseteq N^+(x_i)\cap [1,i]$.
Equivalently, $N^-(x_i)\cap [1,i]\subseteq G_L\cap [1,i]$.
Therefore, $\omega (N^+(x_n)\cap[1,i])\leq \omega (N^+(x_i)\cap[1,i])\leq\omega( N^-(x_i)\cap [1,i])\leq\omega(G_L\cap [1,i])$,
where the second inequality is by the feedback property. Now $L'=x_{p+1}...x_n$ is good also. By induction,
$\omega(N^+(x_n)\cap[p+1,n])\leq \omega (G_{L'})$. Note that $G_{L'}\subseteq G_L\cap [p+1,n]$.
Whence $\omega(N^+(x_n))=\omega (N^+(x_n)\cap[1,i])+\omega(N^+(x_n)\cap[p+1,n])\leq
\omega(G_L\cap [1,i])+\omega (G_{L}\cap [p+1,n])=\omega(G_L)$.
The second part of the statement is obvious.
\end{proof}

Since every (weighted) tournament is a good (weighted) oriented graph, we obtain the following two results.

\begin{cor}(\cite{fidler})
Let L be a weighted median order of a weighted tournament $(T,\omega)$ with feed vertex say f. Then $\omega(N^+(f))\leq\omega(G_L)$.
\end{cor}

\begin{cor}(\cite{m.o.})
Let L be a median order of a tournament with feed vertex say f. Then $|N^+(f))| \leq |G_L|$.
\end{cor}

Let $L$ be a good weighted median order of a good oriented graph $D$ and let $f$ denote its feed vertex. By theorem \ref{prince}, for every
$x\in J(f)$, $\omega(N^+(x)\backslash J(f))\leq\omega(G_L\backslash J(f))$. Let $b_1,\cdots,b_r$ denote the bad vertices of
$L$ not in $J(f)$ and $v_1,\cdots ,v_s$ denote the non bad vertices of $L$ not in $J(f)$, both enumerated in increasing order with respect
to their index in $L$.\\
If $\omega(N^+(f)\backslash J(f))<\omega(G_L\backslash J(f))$, we set $Sed(L)=L$. If $\omega(N^+(f)\backslash J(f))=\omega(G_L\backslash J(f))$,
we set $sed(L)=b_1\cdots b_rJ(f)v_1\cdots v_s$. This new order is called the \textit{sedimentation of $L$}.
\begin{lem}\label{gsedlem}
Let $L$ be a good weighted median order of a good weighted oriented graph $(D,\omega)$. Then $Sed(L)$ is a good weighted median order of $(D,\omega)$.
\end{lem}
\begin{proof}
Let $L=x_1...x_n$ be a good weighted local median order of ($D,\omega$).
If $Sed(L)=L$, there is nothing to prove. Otherwise, we may assume that $\omega(N^+(x_n)\backslash  J(x_n))=\omega(G_L\backslash  J(x_n))$.
The proof is by induction on $r$ the number of bad vertices not in $J(x_n)$. Set $J(x_n)=[x_t,x_n]$.
If $r=0$ then we have $N^-(x_n)\backslash J(x_n)=G_L\backslash  J(x_n)$.
Whence, $\omega(N^+(x_n)\backslash J(x_n))=\omega(G_L\backslash  J(x_n))=\omega(N^-(x_n)\backslash J(x_n))$.
Thus, $Sed(L)=J(x_n)x_1...x_{t-1}$ is a good weighted median order.
Now suppose that $r>0$ and let $i$ be the smallest such that $x_i\notin J(x_n)$ and is bad.
As in the proof of theorem \ref{prince}, $J(x_i)=[x_i,x_p]$ for some $p<n$,
$\omega (N^+(x_n)\cap[1,i])\leq \omega (N^+(x_i)\cap[1,i])\leq\omega( N^-(x_i)\cap [1,i])\leq\omega(G_L\cap [1,i])$ and $\omega(N^+(x_n)\cap[p+1,t-1])\leq \omega (G_{L}\cap[p+1,t-1])$.
However, $\omega(N^+(x_n)\backslash  J(x_n))=\omega(G_L\backslash  J(x_n))$, then the previous inequalities
are equalities. In particular, $\omega (N^+(x_i)\cap[1,i])=\omega( N^-(x_i)\cap [1,i])$. Since $J(x_i)$ is
an interval of $L$ and $D$, then for every $x\in J(x_i)$ we have $\omega (N^+(x)\cap[1,i])=\omega( N^-(x)\cap [1,i])$.
Thus $J(x_i)x_1...x_{i-1}x_{p+1}...x_n$ is a good weighted median order. To conclude, apply the induction hypothesis to the good weighted median order $x_1...x_{i-1}x_{p+1}...x_n$.

\end{proof}

Define now inductively $Sed^{0}(L)=L$ and $Sed^{q+1}(L)$ = $Sed(Sed^{q}(L))$. If the process reaches a rank $q$ such that
$Sed^{q}(L)=y_{1}...y_{n}$ and $\omega (N^{+}(y_{n})\backslash J(y_n))$ $<$ $\omega ( G_{Sed^{q}(L)}\backslash J(y_n) )$, call the order $L$ \textit{stable}.
Otherwise call $L$ \textit{periodic}. These new order are used by Havet and Thomass\'{e} to exhibit a second vertex with the SNP in tournaments
that do not have any sink. We will use them for the same purpose but for other classes of oriented graphs.

\section{Case of oriented graph missing a matching}
In this section, $D$ is an oriented graph missing a matching and $\Delta$ denotes its dependency digraph.
We begin by the following lemma:

\begin{lem}\cite{fidler}\label{structdep}
The maximum out-degree of $\Delta$ is one and the maximum in-degree of $\Delta$ is one. Thus $\Delta$
is composed of vertex disjoint directed paths and directed cycles.
\end{lem}

\begin{proof}
 Assume that $a_{1}b_{1}$ loses to $a_{2}b_{2}$ and $a_{1}b_1$ loses to $a'_{2}b'_{2}$, with $a_1\ra a_2$ and $a_1\ra a'_2$. The edge $a'_2b_2$ is not a missing edge
of $D$. If $a'_2\ra b_2$ then $b_1\ra a'_2\ra b_2$, a contradiction. If $b_2\ra a'_2$ then $b_1\ra b_2\ra a'_2$, a contradiction. Thus,
the maximum out-degree of $\Delta$ is one. Similarly, the maximum in-degree is one.
\end{proof}

In the following, $C=a_1b_1,...,a_kb_k$ denotes a directed cycle of $\Delta$, namely
$a_i\rightarrow a_{i+1}$, $b_{i+1}\notin N^{++}(a_i)\cup N^{+}(a_i)$, $b_i\rightarrow b_{i+1}$
and $a_{i+1}\notin N^{++}(b_i)\cup N^{+}(b_i)$, for all $i<k$. In \cite{fidler}, it is proved that $D[K(C)]$
has a vertex with the SNP. Here we prove that every vertex of $K(C)$ has the SNP in $D[K(C)]$.

\begin{lem}
(\cite{fidler})\label{howlose}
If k is odd then $a_k\rightarrow a_1$, $b_1\notin N^{++}(a_k)\cup N^+(a_k)$,
$b_k\rightarrow b_1$ and $a_1\notin N^{++}(b_k)\cup N^+(b_k)$. If k is even then $a_k\rightarrow b_1$, $a_1\notin N^{++}(a_k)\cup N^+(a_k)$,
$b_k\rightarrow a_1$ and $b_1\notin N^{++}(b_k)\cup N^+(b_k)$.

\end{lem}

\begin{lem}\cite{fidler}\label{intervals}
$K(C)$ is an interval of $D$.
\end{lem}

\begin{proof}
Let $f\notin K(C)$. Then $f$ is adjacent to every vertex in $K(C)$.
If $a_1\ra f$ then $b_2\ra f$, since otherwise $b_2\in N^{++}(a_{1})\cup N^{+}(a_{1})$ which is a contradiction.
So $N^{+}(a_{1})\bks K(C)\sub N^{+}(b_{2})\bks K(C)$. Applying this to every losing relation of $C$ yields
$N^{+}(a_{1})\bks K(C)\sub N^{+}(b_{2})\bks K(C)\sub N^{+}(a_{3})\bks K(C)...\sub N^{+}(b_{k})\bks K(C)\sub N^{+}(b_{1})\bks K(C)
\sub N^{+}(a_{2})\bks K(C)...\sub N^{+}(a_{k})\bks K(C)\sub N^{+}(a_{1})\bks K(C)$ if $k$ is even. So these inclusion are equalities.
An analogous argument proves the same result for odd cycles.
\end{proof}

\begin{lem}
 In $D[K(C)]$ we have:\\
 If $k$ is odd then: $$N^{+}(a_{1})=N^{-}(b_{1})=\{a_{2},b_{3},\cdots,a_{k-1},b_{k}\}$$
$$N^{-}(a_{1})=N^{+}(b_{1})=\{b_{2},a_{3},\cdots,b_{k-1},a_{k}\},$$
If k is even then: $$N^{+}(a_{1})=N^{-}(b_{1})=\{a_{2},b_{3},\cdots,b_{k-1},a_{k}\}$$
$$N^{-}(a_{1})=N^{+}(b_{1})=\{b_{2},a_{3},\cdots,a_{k-1},b_{k}\}.$$
\end{lem}

\begin{proof}
Suppose that $k$ is odd. Set $K:=K(C)$. Then $b_{k}\in N^{+}_{D[K]}(a_{1})$ by lemma \ref{howlose}.
Since $a_{k-1}b_{k-1}$ loses to $a_{k},b_{k}$ and $(a_{1},b_{k})\in E(D)$ then $(a_{1},a_{k-1})\in E(D)$ and so $a_{k-1} \in N^{+}_{D[K]}(a_{1})$, since otherwise $(a_{k-1},a_{1})\in E(D)$ and so $b_{k}\in N^{++}_{D[K]}(a_{k-1})$, which is a contradiction to the definition of the losing relation $a_{k-1}b_{k-1}\ra a_{k}b_{k}$. And so on $b_{k-2},a_{k-3},...,b_{3},a_{2}\in N^{+}_{D[K]}(a_{1})$. Again, since $a_{1}b_{1}$ loses to $a_{2},b_{2}$ then $b_{2}\in N^{-}_{D[K]}(a_{1})$. Since $a_{2}b_{2}$ loses to $a_{3},b_{3}$ and
$(b_{2},a_{1})\in E(D)$ then $(a_{3},a_{1})\ in E(D)$ and so $a_{3} \in N^{-}_{D[K(C)]}(a_{1})$. And so on, $b_{4},a_{5},...,b_{k-1},a_{k}\in N^{-}_{D[K]}(a_{1})$. We use the same argument for finding $N^{+}_{D[K]}(b_{1})$ and $N^{-}_{D[K]}(b_{1})$.
Also we use the same argument when $k$ is even.
\end{proof}

\begin{lem}\label{snp2}

In $D[K(C)]$ we have:
$N^{+}(a_i)=N^{-}(b_i)$, $N^{-}(a_i)=N^{+}(b_i)$,\\
$N^{++}(a_i)=N^{-}(a_i)\cup\{b_i\}\backslash\{b_{i+1}\}$ and $N^{++}(b_i)=N^{-}(b_i)\cup\{a_i\}\backslash\{a_{i+1}\}$
for all $i=1,...,k$ where $a_{k+1}:=a_1$, $b_{k+1}:=b_1$ if $k$ is odd and $a_{k+1}:=b_1$, $b_{k+1}:=a_1$ if $k$ is even.

\end{lem}

\begin{proof}

The first part is due to the previous lemma and the symmetry in these cycles.
For the second part it is enough to prove it for $i=1$ and $a_1$. Suppose first
that $k$ is odd. By definition of losing relation between $a_1b_1$ and $a_2b_2$
we have $b_2\notin N^{++}(a_1)\cup N^{+}(a_1)$. Moreover $a_1\rightarrow a_2\rightarrow b_1$,
whence $b_1\in N^{++}(a_1)$. Note that for $i=1,...,k-1$, $a_i\rightarrow a_{i+1}$ and
$b_i\rightarrow b_{i+1}$. Combining this with the previous lemma we find that
$N^{++}(a_1)=N^{-}(a_1)\cup\{b_1\}\backslash\{b_{2}\}$. Similar argument is used when $k$ is even.

\end{proof}

So we have:

\begin{lem}\label{snp2}
$d^{++}(v)=d^+(v)=d^-(v)=k-1$ for all $v\in K(C)$.
\end{lem}

Let $P=a_{1}b_{1},a_{2}b_{2},\cdots ,a_{k}b_{k}$ be a connected component of $\Delta$, which is also a  maximal path in $\Delta$,
namely $a_i\rightarrow a_{i+1},b_i\rightarrow b_{i+1}$ for $i=1,...,k-1$. Since $a_1b_1$ is a good edge then $(a_1,b_1)$ or $(b_1,a_1)$
is a convenient orientation. If $(a_1,b_1)$ is a convenient orientation, then we orient $(a_i,b_i)$ for $i=1,...,k$. Otherwise, we
orient $a_ib_i$ as $(b_i,a_i)$. We do this for every such a path of $\Delta$. Denote the set of these new arcs by $F$.
Set $D'=D+F$.\\

Since we have oriented all the missing edges of $D$ that form the connected components of $\Delta$ that are paths, then they are no longer missing edges of $D'$ and thus, the dependency digraph of $D'$ is composed of only directed cycles. Then by lemma \ref{intervals} we have:

\begin{lem}
 $D'$ is a good digraph.
\end{lem}

Now, we are ready to prove the following statement:

\begin{thm}
 Every feed vertex of $D'$ has the SNP in D and $D'$.
\end{thm}

\begin{proof}
Let $L$ be a good median order of $D'$ and let $f$ denote its feed vertex.
We have $|N^+_{D'}(f)\backslash J(f)|\leq |G_L^{D'}\backslash J(f)|$ by theorem \ref{prince}.\\

\par Suppose that $f$ is not incident to any new arc of $F$. Then $J(f)=\{f\}$ or $J(f)=K(C)$ ( in $D$ and $D'$) for some cycle $C$
of $\Delta$, $N^+_{D'}(f)=N^+(f)$ and $f$ has the SNP in $D[J(f)]$. Let $y\in N^{++}_{D'}(f)\bks J(f)$. There is a vertex $x$ such that $f\ra x \ra y \ra f$ in $D'$. Note that the arcs $(f, x)$ and $(y, f)$ are in $D$. If $(x,y) \in D$ or is a convenient orientation then $y\in N^{++}(f)$. Otherwise, there is
a missing edge $rs$ that loses to $xy$, namely $s\ra y$ and $x\notin N^{++}(s)\cup N^+(s)$. But $fs$ is not a missing edge then we must have $(f,s)\in D$. Thus $y\in N^{++}(f)$. Hence $N^{++}_{D'}(f)\bks J(f)\sub N^{++}(f)\bks J(f)$. Thus $|N^+(f)|=|N^+_{D'}(f)|= |N^+_{D'}(f)\backslash J(f)| + |N^+_{D'}(f)\cap J(f)| \leq
|G_L^{D'}\backslash J(f)| + |N^{++}_{D'}(f)\cap J(f)| \leq |N^{++}(f)\bks J(f)| + |N^{++}_{D[J(f)]}(f)|= |N^{++}(f)|$.\\

\par Suppose that $F$ is incident to a new arc of $F$. Then there is a path $P=a_{1}b_{1},a_{2}b_{2},\cdots ,a_{k}b_{k}$ in $\Delta$, which is also a connected component $\Delta$, namely $a_t\rightarrow a_{t+1},b_t\rightarrow b_{t+1}$ for $t=1,...,k-1$, such that $f=a_i$ or $f=b_i$. We may suppose without loss of generality that $(a_t,b_t)\in D',$ $ \forall t\in \{1,...,k\}$. Suppose first that $f=a_i$ and $i<k$.
Then $f$ gains only $b_i$ as a first out-neighbor and $b_{i+1}$ as a second out-neighbor. Indeed, let
$y\in N^{++}_{D'}(f)\backslash \{b_{i+1}\}$. There is a vertex $x$ such that $f\ra x \ra y \ra f$ in $D'$. Suppose that $b_i\neq x$. Note that the arcs $(f, x)$ and $(y, f)$ are in $D$. If $(x,y) \in D$ or is a convenient orientation then $y\in N^{++}(f)$. Otherwise, there is
a missing edge $rs$ that loses to $xy$, namely $s\ra y$ and $x\notin N^{++}(s)\cup N^+(s)$. But $fs$ is not a missing edge then we must have $(f,s)\in D$. Thus $y\in N^{++}(f)$. Suppose that $b_i=x$. Since $b_i\ra y$, $a_{i+1}\notin N^{++}(b_i)\cup N^+(b_i)$ and $a_{i+1}y$ is not a missing edge, then we must have $(y,a_{i+1})\in D$. Thus $f\ra a_{i+1}\ra y$ in $D$ and $y\in N^{++}(f)$. Hence $N^{++}_{D'}(f)\backslash \{b_{i+1}\} \sub N^{++}(f)$. Note that $J(f)=\{f\}$ in $D'$. Combining this with theorem \ref{prince}, we get $|N^{+}(f)| = |N^{+}_{D'}(f)|-1 \leq |N^{++}_{D'}(f)|-1 \leq |N^{++}(f)|$.
Now suppose that $f=a_k$. We reorient the missing edge $a_kb_k$ as $(b_k, a_k)$ and let $D''$ denote the new oriented graph. Then $L$ is a good median order of the good oriented graph $D''$,
$N^+_{D''}(f)=N^+(f)$, $J(f)=\{f\}$ in $D''$, and $f$ has the SNP in $D''$. Let $y\in N^{++}_{D''}(f)$. There is a vertex $x$ such that $f\ra x \ra y \ra f$ in $D''$. Note that the arcs $(f, x)$ and $(y, f)$ are in $D$. If $(x,y) \in D$ or is a convenient orientation then $y\in N^{++}(f)$. Otherwise, there is
a missing edge $rs$ that loses to $xy$, namely $s\ra y$ and $x\notin N^{++}(s)\cup N^+(s)$. But $fs$ is not a missing edge then we must have $(f,s)\in D$. Thus $y\in N^{++}(f)$ and $N^{++}_{D''}(f)\sub N^{++}(f)$.
Thus $f$ has the SNP in $D$.
Finally, suppose that $f=b_i$. We use the same argument of the case $f=a_k$ to prove that $f$ has the SNP in $D$.
\end{proof}

We note that our method guarantees that the vertex $f$ found with the SNP is a feed vertex of some digraph containing $D$.
This is not guaranteed by the proof presented in \cite{fidler}.
 Recall that $F$ is the set of the new arcs added to $D$ to obtain
the good oriented graph $D'$. So if $F=\phi$ then $D$ is a good oriented graph.

\begin{thm}
Let $D$ be an oriented graph missing a matching and suppose that $F=\phi$. If $D$ has no sink vertex
then it has at least two vertices with the SNP.
\end{thm}

\begin{proof}

Consider a good median order $L=x_1...x_n$ of $D$. If $J(x_n)=K(C)$ for some directed cycle $C$ of $\Delta$ then by lemma \ref{prince}
and lemma \ref{snp2} the result holds. Otherwise, $x_n$ is a whole vertex (i.e. $J(x_n)=\{x_n\}$). By lemma \ref{prince}, $x_n$ has the SNP in $D$.
So we need to find another vertex with SNP. Consider the good median order $L'=x_1...x_{n-1}$. Suppose first that $L'$
is stable. There is $q$ for which $Sed^q(L')=y_1...y_{n-1}$ and
$\mid N^+(y_{n-1})\backslash J(y_{n-1})\mid < \mid G_{Sed^q(L')}\backslash J(y_{n-1})\mid$.
Note that $y_1...y_{n-1}x_n$ is also a good median order of $D$.
By lemma \ref{prince} and lemma \ref{snp2}, $y:=y_{n-1}$ has the SNP in $D[y_1,y_{n-1}]$.
So $\mid N^+(y)\mid = \mid N^+_{D[y_1,y_{n-1}]}(y)\mid +1\leq \mid G_{Sed^q(L')}\mid \leq \mid N^{++}(y)\mid $.
Now suppose that $L'$ is periodic. Since $D$ has no sink then $x_n$ has an out-neighbor $x_j$. Choose $j$ to be the greatest (so that it is the last
vertex of its corresponding interval).
Note that for every $q$, $x_n$ is an out-neighbor of the feed vertex of $Sed^q(L')$.
So $x_j$ is not the feed vertex of any $Sed^q(L')$.
Since $L'$ is periodic, $x_j$ must be a bad vertex of $Sed^q(L')$ for some integer $q$, otherwise the index of $x_j$ would
always increase during the sedimentation process. Let $q$ be such an integer.
Set $Sed^q(L')=y_1...y_{n-1}$. Lemma \ref{snp2} and lemma \ref{prince} guarantee that the vertex $y:=y_{n-1}$ with the SNP in $D[y_1,y_{n-1}]$.
Note that $y\rightarrow x_n \rightarrow x_j$ and $G_{Sed^q(L')}\cup \{x_j\}\subseteq N^{++}(y)$.
So $\mid N^+(y)\mid = \mid N^+_{D[y_1,y_{n-1}]}(y) \mid +1 = \mid G_{Sed^q(L')} +1 \mid= \mid G_{Sed^q(L')}\cup \{x_j\} \mid
\leq \mid N^{++}(y) \mid$.

\end{proof}

\end{document}